\documentclass[10pt, a4paper, reqno]{amsart}
\usepackage{amssymb,amsmath,amsthm}
\usepackage[T1]{fontenc}
\usepackage[dvips]{graphicx}

\theoremstyle{plain}
\newtheorem{theorem}{Theorem}

\newtheorem{proposition}{Proposition}
\newtheorem*{theorem A}{Theorem A}
\newtheorem*{theorem B}{Theorem B}
\newtheorem*{theorem C}{Theorem C}
\theoremstyle{definition}
\newtheorem{definition}{Definition}
\newtheorem{rem}{Remark}

\newtheorem*{proof A}{Proof of Theorem A}
\newtheorem*{proof B}{Proof of Theorem B}

\DeclareMathOperator*{\hd}{\mathrm{dim}_\mathrm{H}}

\begin{document}

\title{Measure of full dimension for some nonconformal repellers}
\maketitle \centerline{\small{NUNO LUZIA}}
\footnotetext{
 Financed by Fundação para a Ciência e a Tecnologia (Portugal)}
\vspace{24pt}
\centerline{\scriptsize ABSTRACT} \text{}\\
\begin{footnotesize}
Given $(X,T)$ and $(Y,S)$ mixing subshifts of finite type such that $(Y,S)$ is a factor of $(X,T)$
with factor map $\pi\colon X\to Y$, and positive Hölder continuous functions
$\varphi\colon X\to \mathbb{R}$ and $\psi\colon Y\to \mathbb{R}$, we prove that the maximum of
\[
   \frac{h_{\mu\circ \pi^{-1}}(S)}{\int \psi\circ\pi\,d\mu}+
   \frac{h_\mu(T)-h_{\mu\circ \pi^{-1}}(S)}{\int \varphi\,d\mu}
\]
over all $T$-invariant Borel probability measures $\mu$ on $X$ is attained on the subset
of ergodic measures. Here $h_\mu(T)$ stands for the metric entropy of $T$ with respect to $\mu$.
As an application, we prove the existence of an ergodic invariant measure with full dimension for
a class of transformations treated in \cite{GP1}, and also for the transformations treated in
\cite{L1}, where the author considers nonlinear skew-product perturbations of
\emph{general Sierpinski carpets}. In order to do so we establish a variational principle for the
topological pressure of certain noncompact sets.
\end{footnotesize}

\section{Introduction}
The problem we are interested in is the computation of Hausdorff
dimension of invariant sets for dynamical systems, and moreover
the existence of an ergodic invariant measure on
the invariant set with full dimension. We are particularly
interested in the case of a map $f$ and $f$-invariant sets
$\Lambda$ such that $f|\Lambda$ is \emph{expanding}. Even in this
case, it is still a widely open problem only solved in great
generality when $f|\Lambda$ is expanding and conformal, due to the
thermodynamic formalism introduced by Sinai, Ruelle and Bowen (see
\cite{B2}, \cite{B3}, \cite{R2}): there is a formula in terms of the \emph{pressure
function} for the Hausdorff dimension, and there is an ergodic
invariant measure of full dimension which is a \emph{Gibbs state}.

In the nonconformal setting the computation of Hausdorff dimension is more delicate due to the
existence of several rates of expansion, and no such general formula exists. In \cite{L1}
we consider a class of skew-product expanding maps of the 2-torus of the form
$f(x,y)=(a(x,y), b(y))$ satisfying $\inf |Da|> \sup |Db|$, and
consider invariant sets $\Lambda$ which possess a simple Markov
structure, proving the \emph{variational principle for Hausdorff dimension}, i.e the existence of
ergodic invariant measures on $\Lambda$ with dimension arbitrarily close to the dimension of
$\Lambda$. Now because of the nonlinearity of the maps the problem of finding a maximizing
ergodic measure turns into another nontrivial problem. This is
because although we have a sequence of ergodic invariant measures
whose dimension converges to the dimension of the invariant set,
the map $\mu\mapsto\hd\mu$ is not in
general upper-semicontinuous. In \cite{L2} we exhibit an open class of
repellers, that possess a \emph{dominating spliting} that is
\emph{not too strong}, for which there is an ergodic measure of
full dimension. This is done by showing that the methods of Heuter
and Lalley \cite{HL}, who prove the validity of Falconer's formula (see
\cite{F2}) for an open class of self-affine transformations, extend to
nonlinear transformations. In this work we prove the existence of
an ergodic invariant measure with full dimension for the class of
transformations treated in \cite{L1}, thus giving new insight for
solving this kind of problem for nonlinear transformations. The
novelty is in reducing the problem to the the variational principle
for the topological pressure for certain noncompact sets, which we
establish.

Now we give an overview of related results, culminating in \cite{L1}
and this work. The simplest examples of invariant sets for
nonconformal transformations are the \emph{general Sierpinski
carpets}. These are invariant sets for transformations of the
2-torus of the form $f(x,y)=(lx,my)$, where $l>m>1$ are integers,
and are modeled by a full shift. Bedford \cite{Be} and McMullen \cite{Mu}
proved independently a formula for the Hausdorff dimension of
these sets which coincides with the dimension of a Bernoulli
measure. Later, Gatzouras and Lalley \cite{GL} generalized these
results for certain self-affine sets. Until then only sets modeled
by a full shift were treated when Kenyon and Peres \cite{KP} considered
expanding linear endomorphisms of the $d$-torus which are a direct
sum of conformal endomorphisms, and proved that any invariant set
supports an ergodic invariant measure of full dimension, thus
solving the problem for these kind of transformations. Following
\cite{GL}, Gatzouras and Peres \cite{GP1} considered nonlinear maps of the
form $f(x,y)=(f_1(x),f_2(y))$, where $f_1$ and $f_2$ are conformal
and expanding maps satisfying $\inf |Df_1|\ge\sup |Df_2|$, and
showed that, for a large class of invariant sets $\Lambda$, there
exist ergodic invariant measures with dimension arbitrarily close
to the dimension of $\Lambda$. The methods used in this work also provide an invariant
ergodic measure with full dimension for a class of tranformations treated in \cite{GP1}.

The notion of topological entropy for noncompact sets was
introduced by Bowen \cite{B1}. In \cite{PP} Pesin and Pitskel introduced the
notion of topological pressure for noncompact sets and proved the
variational principle for the topological pressure for certain noncompact sets. The noncompact
sets we are interested in are the ones treated in \cite{BS} and \cite{TV} where
they prove the variational principle for the topological entropy
for these sets.

Let us now state the main result. Given a continuous map $T\colon X\to X$ of a compact metric
space $X$, we use the following notation: $\mathcal{M}(T)$ is the set of all $T$-invariant Borel
probability measures on $X$; $\mathcal{M}_e(T)\subset \mathcal{M}(T)$ is the subset of ergodic
measures; $h_\mu(T)$ is the metric entropy of $T$ with respect to $\mu\in\mathcal{M}(T)$.

\begin{theorem A}
Let $(X,T)$ and $(Y,S)$ be mixing subshifts of finite type, and $\pi\colon X\to Y$ be a
continuous and surjective mapping such that $\pi\circ T=S\circ \pi$ ($S$ is a factor of $T$).
Let $\varphi\colon X\to \mathbb{R}$ and $\psi\colon Y\to \mathbb{R}$ be
positive Hölder continuous functions. Then the maximum of
\begin{equation}\label{pvprinc}
   \frac{h_{\mu\circ \pi^{-1}}(S)}{\int \psi\circ\pi\,d\mu}+
   \frac{h_\mu(T)-h_{\mu\circ \pi^{-1}}(S)}{\int \varphi\,d\mu}
\end{equation}
over all $\mu\in\mathcal{M}(T)$ is attained on the set $\mathcal{M}_e(T)$.
\end{theorem A}

\begin{rem}
Theorem A answers positively to Problem 2 raised in \cite{GP2} for mixing subshifts of finite type and
Hölder continuous potentials. So, it applies to obtain an invariant ergodic measure
of full dimension for a class of transformations treated in \cite{GP1}.
\end{rem}

Now we define \emph{general Sierpinski carpet}. Let
$\mathbb{T}^2=\mathbb{R}^2/\mathbb{Z}^2$ be the 2-torus and
$f_0\colon \mathbb{T}^2\to \mathbb{T}^2$ be given by
$f_0(x,y)=(lx, my)$ where $l>m>1$ are integers. The grid of lines
$[0,1]\times\{i/m\},\, i=0,...,m-1$, and $\{j/l\}\times[0,1],\,
j=0,...,l-1$, form a set of rectangles each of which is mapped by
$f_0$ onto the entire torus (these rectangles are the domains of
invertibility of $f_0$). Now choose some of these rectangles and
consider the fractal set $\Lambda_0$ consisting of those points
that always remain in these chosen rectangles when iterating
$f_0$. Geometrically, $\Lambda_0$ is the limit (in the Hausdorff
metric) of \emph{$n$-approximations}: the 1-approximation consists
of the chosen rectangles, the 2-approximation consists in dividing
each rectangle of the 1-aproximation into $l\times m$
subrectangles and selecting those with the same pattern as in the
begining, and so on. We say that $(f_0,\Lambda_0)$ is a \emph{general Sierpinski carpet}.

Let $(f_0,\Lambda_0)$ be a general Sierpinski carpet. There exists
$\varepsilon>0$ such that if $f$ is  $\varepsilon-\mathrm{C}^1$
close to $f_0$ then there is a unique homeomorphism
$h\colon\mathbb{T}^2\to \mathbb{T}^2$ close to the identity which
conjugates $f$ and $f_0$, i.e $f\circ h=h\circ f_0$ (see \cite{S}).

\begin{definition}
The $f$-invariant set $\Lambda=h(\Lambda_0)$ is called the
$f$-\emph{continuation} of $\Lambda_0$.
\end{definition}

\begin{definition}
Let $\mathcal{S}$ be the class of $\mathrm{C}^2$ maps $f\colon
\mathbb{T}^2\to \mathbb{T}^2$ of the form
\[
   f(x,y)=(a(x,y),b(y)).
\]
\end{definition}

\emph{Notation}: $\hd\Lambda$ and $\hd\mu$ stand for the Hausdorff
dimension of a set $\Lambda$ and a measure $\mu$, respectively.

Then the following is proved in \cite{L1}.

\begin{theorem}\label{teol}
Let $(f_0,\Lambda_0)$ be a general Sierpinski carpet. There exists
$\varepsilon>0$ such that if $f\in\mathcal{S}$ is
$\varepsilon-\mathrm{C}^2$ close to $f_0$, and $\Lambda$ is the
$f$-continuation of $\Lambda_0$, then
\[
 \hd\Lambda=\sup\{\hd\mu : \mu(\Lambda)=1, \text{ $\mu$ is $f$-invariant and ergodic}\}.
\]
\end{theorem}

Here we improve this result.

\begin{theorem B}
With the same hypothesis of Theorem \ref{teol}, there exists an
ergodic invariant measure $\mu$ on $\Lambda$ such that
\[
 \hd\Lambda=\hd\mu .
\]
Morever, $\mu$ is a Gibbs state for a relativized variational
principle.
\end{theorem B}

To prove these results we use the variational principle for the
topological pressure for certain noncompact sets as we shall
describe now. Given a continuous map $T\colon X\to X$ of a
compact metric space,denote by $P(\psi, K)$ the topological pressure associated to a continuous
function $\psi\colon X\to\mathbb{R}$ and a $T$-invariant set $K$
(not necessarly compact), as defined in \cite{PP}. Let
\[
    I_\psi=\left(\inf_{\mu\in\mathcal{M}(T)} \int \psi\,d\mu ,
    \sup_{\mu\in\mathcal{M}(T)} \int \psi\,d\mu\right) .
\]

\begin{theorem C}
Let $(X, T)$ be a mixing subshift of finite type, and
$\varphi,\psi\colon X\to\mathbb{R}$ Hölder continuous functions.
For $\alpha\in\mathbb{R}$ let
\[
     K_\alpha=\left\{x\in X : \lim_{n\to\infty} \frac{1}{n}
     \sum_{i=0}^{n-1} \psi(T^i(x))=\alpha\right\}.
\]
If $0\notin \partial I_\psi$ and $\alpha\in I_\psi$ then
\[
   P(\varphi, K_\alpha)=\sup\left\{ h_\mu(T)+\int \varphi\,d\mu:
   \mu\in\mathcal{M}(T) \text{ and } \int \psi\,d\mu=\alpha \right\}.
\]
Morever, the supremum is attained at a unique measure $\mu_\beta$ which is the Gibbs state
with respect to the potential $\varphi+\beta \psi$, for a unique $\beta\in\mathbb{R}$.
\end{theorem C}

For definitions and basic results about dimension we refer the
reader to the books \cite{F1} and \cite{P}.

\section{Proof of Theorem C}

We start by proving the `moreover' part. We have, for all $\beta\in\mathbb{R}$,
\begin{align*}
&\sup\left\{ h_\mu(T)+\int \varphi\,d\mu: \mu\in\mathcal{M}(T) \text{
and }
  \int \psi\,d\mu=\alpha \right\}\\
&=\sup\left\{ h_\mu(T)+\int \varphi\,d\mu + \beta \int \psi\,d\mu:
\mu\in\mathcal{M}(T) \text{ and }
  \int \psi\,d\mu=\alpha \right\}-\beta\alpha\\
&\le\sup\left\{ h_\mu(T)+\int \varphi\,d\mu + \beta \int
\psi\,d\mu: \mu\in\mathcal{M}(T) \right\}
  -\beta\alpha.
\end{align*}
Now it is well known (see \cite{B2}) that the last supremum is uniquely attained
at the \emph{Gibbs state} $\mu_\beta$ associated to the potential
$\varphi + \beta \psi$ (for the classical variational principle).
So we must find a unique $\beta$ such that $\int \psi\,d\mu_\beta=\alpha$.

We use the abbreviation $P(\cdot)=P(\cdot,X)$. It is proved in
\cite{R1} that $P(\cdot)$ is a real analytic function on the space of
Hölder continuous functions and that
\begin{align*}
 &\frac{d}{d\varepsilon}\Bigr|_{\varepsilon=0}P(h_1+\varepsilon h_2)=
  \int h_2\,d\mu_{h_1},\\
 &{\frac{\partial^2P(h+\varepsilon_1 h_1+\varepsilon_2 h_2)}
  {\partial\varepsilon_1\partial\varepsilon_2}}\Big|_{\varepsilon_1=\varepsilon_2=0}
  =Q_h(h_1,h_2),
\end{align*}
where $Q_h$ is the bilinear form defined by
\begin{equation}\label{forma}
 Q_h(h_1,h_2)=\sum_{n=0}^\infty \left( \int h_1(h_2\circ f^n)\,d\mu_h - \int h_1\,d\mu_h
 \int h_2\,d\mu_h \right),
\end{equation}
and $\mu_h$ is the Gibbs measure for the potential $h$. Moreover,
$Q_h(h_1,h_1)\ge0$ and $Q_h(h_1,h_1)=0$ if and only if $h_1$ is
\emph{cohomologous} to a constant function. From this we get that
\begin{equation}\label{monot}
   \frac{d}{d\beta}\int \psi\,d\mu_\beta=\frac{d^2}{d\beta^2} P(\varphi+\beta \psi)=
   Q_{\varphi+\beta \psi}(\psi,\psi)>0
\end{equation}
(the hypothesis $\alpha\in I_\psi$ prevents $\psi$ being
cohomologous to a constant). So we must see that
\begin{align}
  \lim_{\beta\to\infty} \int \psi\,d\mu_\beta &=\sup_{\mu\in\mathcal{M}(T)}
   \int \psi\,d\mu, \label{ext1}\\
  \lim_{\beta\to -\infty} \int \psi\,d\mu_\beta &=\inf_{\mu\in\mathcal{M}(T)}
   \int \psi\,d\mu. \label{ext2}
\end{align}
\emph{Proof of (\ref{ext1}):} We use the notation
\[
   p(\beta)\sim q(\beta)\quad (\beta\to\infty) \quad\text{means}\quad
   \lim_{\beta\to\infty}\frac{p(\beta)}{q(\beta)}=1.
\]
We have that
\begin{equation}\label{cauchy1}
  \int \psi\,d\mu_\beta= \frac{d}{d\beta} P(\varphi+\beta\psi)=\frac{d}{d\beta}
  \sup_{\mu\in\mathcal{M}(T)} \left\{ h_\mu(T)+\int\varphi\,d\mu + \beta \int\psi\,d\mu\right\}.
\end{equation}
Since $\mu\mapsto h_\mu(T)+\int\varphi\,d\mu$ is bounded, it is
easy to see that
\begin{equation}\label{cauchy2}
 \sup_{\mu\in\mathcal{M}(T)} \left\{ h_\mu(T)+\int\varphi\,d\mu + \beta \int\psi\,d\mu\right\}
 \sim \beta \sup_{\mu\in\mathcal{M}(T)} \int \psi\,d\mu,
\end{equation}
so, using L'Hospital's rule applied to (\ref{cauchy1}) and
(\ref{cauchy2}), we obtain (\ref{ext1}). The proof of (\ref{ext2})
is similar.

Now we prove equality with $P(\varphi, K_\alpha)$. As before, let
$\mu_\beta$ be the Gibbs state for the potential
$\varphi+\beta\psi$ where $\beta$ is such that $\int
\psi\,d\mu_\beta=\alpha$. Since $\mu_\beta$ is ergodic, by
Birkhoff's ergodic theorem we get that $\mu_\beta(K_\alpha)=1$ and
this implies that (see \cite{PP})
\[
  P(\varphi,K_\alpha)\ge h_{\mu_\beta}(f)+\int \varphi\,d\mu_\beta.
\]
So we are left to prove
\[
  P(\varphi,K_\alpha)\le P(\varphi+\beta\psi)-\beta\alpha.
\]
Given $x\in X$ we use the notation $x=(x_1,x_2,...,x_n,...)$ and
denote by $[x_1,...,x_n]$ the cylinder set $\{y\in X :
y_i=x_i,\,i=1,...,n\}$. Fix $\varepsilon>0$ and, for each $N\in\mathbb{N}$, define $n_N(x)$,
for $x\in K_\alpha$, to be the least integer $\ge N$ for which
\begin{equation}\label{pres1}
    \left|\frac{1}{n_N(x)}\sum_{i=0}^{n_N(x)-1}\psi(f^i(x))-\alpha\right|<\varepsilon
\end{equation}
is satisfied. Then let $\tilde{\mathcal{C}}^{(\varepsilon)}_N$ consist of all cylinders
$[x_1,...,x_{n_N(x)}]$ for $x\in K_\alpha$. Notice that $\tilde{\mathcal{C}}^{(\varepsilon)}_N$
is a countable cover of $K_\alpha$, because $n_N(x)<\infty$ for each $x\in K_\alpha$ and because
there are only a countable number of cylinders (of finite length) to begin with. Now take
$\mathcal{C}^{(\varepsilon)}_N$ to be a subcover of $\tilde{\mathcal{C}}^{(\varepsilon)}_N$
consisting of pairwise disjoint cylinders, which exists simply because cylinders are nested.
For each $C\in \mathcal{C}^{(\varepsilon)}_N$ one has that $C=[x_1^C,...,x_{n_N(x^C)}^C]$ for
some $x^C\in K_\alpha$; fixing such an $x^C$ for each $C$, we define, for $\lambda\in\mathbb{R}$,
\[
  m(\varphi,\lambda, \mathcal{C}^{(\varepsilon)}_N)=\sum_{C\in\mathcal{C}^{(\varepsilon)}_N}
  \exp(-\lambda n_N(x^C)+\overline{S}_{n_N(x^C)}\varphi(x^C)),
\]
where
\[
    \overline{S}_{n}\varphi(x)=\sup \left\{ \sum_{i=0}^{n-1}\varphi(f^i(y)) :
    y_i=x_i,\,i=1,...,n \right\}.
\]
Then, it follows from the definition of topological pressure (see
\cite{PP}) that
\[
      \sup_{N\in\mathbb{N}} m(\varphi,\lambda, \mathcal{C}^{(\varepsilon)}_N)<\infty
      \Longrightarrow P(\varphi, K_\alpha)\le \lambda.
\]
Now, since $\mu_\beta$ is the Gibbs state for the potential
$\varphi+\beta\psi$, there are positive constants $c_1$ and $c_2$ (independent of $\varepsilon$
and $N$) such that, for every $C\in\mathcal{C}^{(\varepsilon)}_N$,
\[
   c_1\le\frac{\mu_\beta([x_1^C...x_{n_N(x^C)}^C])}
   {\exp(-P(\varphi+\beta\psi)n_N(x^C)+S_{n_N(x^C)}\varphi(x^C)+\beta S_{n_N(x^C)}\psi(x^C))}
   \le c_2,
\]
where $S_{n}\varphi(x)=\sum_{i=0}^{n-1}\varphi(f^i(x))$ (see \cite{B2}). Summing over
$C\in\mathcal{C}^{(\varepsilon)}_N$ we obtain
\begin{equation}\label{pres2}
 \sum_{C\in\mathcal{C}^{(\varepsilon)}_N} \exp(-P(\varphi+\beta\psi)n_N(x^C)+S_{n_N(x^C)}
 \varphi(x^C)+\beta S_{n_N(x^C)}\psi(x^C))\le c_1^{-1}.
\end{equation}
Since $\varphi$ is Hölder continuous, by a classical argument of
bounded distortion, there exists a constant $c$ such that, for
every $n\in\mathbb{N}$ and $x\in X$,
\begin{equation}\label{pres3}
   |S_n\varphi(x)-\overline{S}_{n}\varphi(x)|\le c.
\end{equation}
So, using (\ref{pres1}) and (\ref{pres3}) in (\ref{pres2}) we obtain
\begin{equation*}
 m(\varphi, P(\varphi+\beta\psi)-\beta\alpha+|\beta|\varepsilon, \mathcal{C}^{(\varepsilon)}_N)
 \le e^c c_1^{-1} \quad \forall\, N\in\mathbb{N},
\end{equation*}
which shows that $P(\varphi, K_\alpha)\le P(\varphi+\beta\psi)-\beta\alpha+|\beta|\varepsilon$.
Since $\varepsilon$ is arbitrarily small we get
\[
  P(\varphi, K_\alpha)\le P(\varphi+\beta\psi)-\beta\alpha
\]
which finishes the proof of Theorem C.

\section{Main result}

Let $T\colon X\to X$ and $S\colon Y\to Y$ be continous mappings of compact metric spaces, and
$\pi\colon X\to Y$ be a continuous and surjective mapping such that $\pi\circ T=S\circ \pi$.
Then the \emph{relativized variational principle} (see \cite{LW}) says that,
given $\nu\in\mathcal{M}(S)$ and a continuous function $\varphi\colon X\to \mathbb{R}$,
\begin{equation}\label{rpv}
  \sup_{\substack{\mu\in\mathcal{M}(X)\\ \mu\circ\pi^{-1}=\nu}}
  \left\{h_\mu(T)-h_\nu(S)+\int_X \varphi\,d\mu\right\}=\int_{Y}
  P(T,\varphi,\pi^{-1}(y))\,d\nu(y),
\end{equation}
where $P(T,\varphi,Z)$ denotes the relative pressure
of $T$ with respect to $\varphi$ and a compact set $Z\subset X$. We say that $\mu$ is an
\emph{equilibrium state} for (\ref{rpv}) if the supremum is attained at $\mu$.

Now if, moreover, $(X,T)$ and $(Y,S)$ are mixing subshifts of finite type then,
according to \cite{DG}, \cite{DGH}, there is a unique equilibrium state $\mu$ for (\ref{rpv})
relative to any $\nu\in\mathcal{M}(S)$ and any Hölder continuous $\varphi$. Moreover,
$\mu$ is ergodic if $\nu$ is ergodic.\\

\emph{Proof of Theorem A}.
Since $\varphi$ is positive then, given $\nu\in\mathcal{M}(S)$, there is a unique
real $t(\nu)\in [0,1]$ such that
\begin{equation}\label{pres0}
    \int_{Y} P(T,-t(\nu)\varphi,\pi^{-1}(y))\,d\nu(y)=0
\end{equation}
(note that $t\mapsto P(T,-t\varphi,\pi^{-1}(y))$ is strictly
decreasing). Denote by $\mu_\nu$ the unique equilibrium state for
(\ref{rpv}) relative to $\nu$ and $-t(\nu)\varphi$. Then it
follows from the relativized variational principle that, for $\mu\in\mathcal{M}(T)$ such that
$\mu\circ\pi^{-1}=\nu$,
\begin{equation}\label{restrain}
   \frac{h_\mu(T)-h_\nu(S)}{\int \varphi\,d\mu}\le t(\nu)
\end{equation}
with equality if and only if $\mu=\mu_\nu$. Put
\[
  D(\mu)= \frac{h_{\mu\circ \pi^{-1}}(S)}{\int \psi\circ\pi\,d\mu}+
   \frac{h_\mu(T)-h_{\mu\circ \pi^{-1}}(S)}{\int \varphi\,d\mu}
\]
and
\begin{equation}\label{pvorig1}
   D=\sup_{\mu\in\mathcal{M}(T)} D(\mu).
\end{equation}
Then it follows by (\ref{restrain}) that
\begin{equation}\label{pvorig2}
   D=\sup_{\nu\in\mathcal{M}(S)} \left\{
   \frac{h_{\nu}(S)}{\int \psi \,d\nu}+ t(\nu) \right\},
\end{equation}
and if this supremum is attained at $\nu_0\in\mathcal{M}_e(S)$ then the supremum in
(\ref{pvorig1}) is attained at $\mu_{\nu_0}\in\mathcal{M}_e(T)$ as we wish.

It follows from (\ref{pvorig2}) that
\begin{equation}\label{novopv1}
   \sup_{\nu\in\mathcal{M}(S)} \left\{ h_\nu(S)+(t(\nu)-D)\int \psi\,d\nu \right\}=0,
\end{equation}
and if this supremum is attained at $\nu_0\in\mathcal{M}_e(S)$ then so is the supremum in
(\ref{pvorig2}) and thus the supremum in (\ref{pvorig1}) is attained at
$\mu_{\nu_0}\in\mathcal{M}_e(T)$ as we wish.
Let
\[
  \underline{t}=\inf_{\nu\in\mathcal{M}(S)} t(\nu) \quad\text{and}\quad
  \overline{t}=\sup_{\nu\in\mathcal{M}(S)} t(\nu).
\]
If $\underline{t}=\overline{t}$ then
\[
     D=\sup_{\nu\in\mathcal{M}(S)} \frac{h_\nu(S)}{\int \psi\,d\nu}+\overline{t}
     =\frac{h_{\nu_0}(S)}{\int \psi\,d\nu_0}+t(\nu_0),
\]
where $\nu_0$ is the \emph{Gibbs state} for the potential $-D\psi$ which is well known to be
ergodic (see \cite{B3}). Otherwise, the supremum in (\ref{novopv1}) can be
rewritten as
\begin{equation}\label{novopv2}
   \sup_{\underline{t}\le t\le\overline{t}}\; \sup_{\substack{\nu\in\mathcal{M}(S) \\ t(\nu)=t}}
   \left\{ h_\nu(S)+\int (t-D)\psi\,d\nu \right\}.
\end{equation}
According to \cite{DG}, \cite{DGH}, there is a Hölder continuous function
$A_{-t\varphi}\colon Y\to\mathbb{R}$ such that
\begin{equation}\label{gauge}
 \int \log A_{-t\varphi}\,d\nu=\int P(T,-t\varphi,\pi^{-1}(y))\,d\nu(y),
\end{equation}
so by (\ref{pres0}),
\[
    t(\nu)=t \Leftrightarrow \int \log A_{-t\varphi}\,d\nu=0.
\]
So, the supremum in (\ref{novopv2}) can be rewritten as
\begin{equation}\label{novopv3}
   \sup_{\underline{t}\le t\le \overline{t}}\sup \left\{ h_\nu(S)+\int (t-D)\psi\,d\nu :
   \nu\in\mathcal{M}(S) \text{ and } \int \log A_{-t\varphi}\,d\nu=0 \right\}.
\end{equation}
Now we see that we satisfy the hypotheses of Theorem B.
The case for which this supremum is attained at $\underline{t}$ or $\overline{t}$ will be
treated in the end of the proof, so we assume this does not occur.
If $t\in (\underline{t},\overline{t})$ then there exists
$\overline{\nu}\in\mathcal{M}(S)$ such that $t(\overline{\nu})<t$
which implies that
\begin{align*}
  &\inf_{\nu\in\mathcal{M}(S)} \int \log A_{-t\varphi}\,d\nu\le
  \int \log A_{-t\varphi}\,d\overline{\nu}=\\
  &=\int P(T,-t\varphi,\pi^{-1}(y))\,d\overline{\nu}(y)<
  \int P(T,-t(\overline{\nu})\varphi,\pi^{-1}(y))\,d\overline{\nu}(y)=0.
\end{align*}
In the same way we get that
\[
   \sup_{\nu\in\mathcal{M}(S)} \int \log A_{-t\varphi}\,d\nu>0.
\]
So, applying Theorem B we get that the intermediate supremum in
(\ref{novopv3}) is attained at the Gibbs measure (hence ergodic)
$\nu_{\beta(t)}$ for the potential $(t-D)\psi+\beta(t)\log A_{-t\varphi}$, and the value of this
supremum is, with $P(\cdot)=P(\cdot, Y)$,
\begin{equation}\label{h}
   h(t)=P((t-D)\psi+\beta(t)\log A_{-t\varphi}),
\end{equation}
where $\beta(t)$ is the unique real satisfying
\[
    \int \log A_{-t\varphi}\,d\nu_{\beta(t)}=0\quad\text{i.e.}\quad t(\nu_{\beta(t)})=t.
\]
We shall see that the function $(\underline{t},\overline{t})\ni
t\mapsto h(t)$ is continuous.

We observe the following fact from \cite{DG} (see Corollary 4.14,
Remark 4.16 and Proposition 5.5): If
$\phi\colon X\to\mathbb{R}$ is Hölder continuous then there
exist continuous functions
$A^n_\phi\colon Y\to\mathbb{R}$, $n\in\mathbb{N}$, such
that $A^n_\phi\to A_\phi$ (in the uniform topology) and
$\phi\mapsto A^n_\phi$ is continuous. Moreover, if $\phi$
satisfies
\begin{equation}\label{holder}
   |\phi(x_1)-\phi(x_2)|\le \alpha d(x_1,x_2)^\gamma,\quad x_1,x_2\in X,
\end{equation}
then the speed of convergence of $A^n_\phi\to A_\phi$ is
exponential depending only on $\alpha$ and $\gamma$. This implies
that $[0,1]\ni t\mapsto A_{-t\varphi}$ is continuous.

Now we prove continuity of $(\underline{t},\overline{t})\ni
t\mapsto \beta(t)$. Let $\mathcal{H}^{\alpha,\gamma}$ be the set
of functions $\phi\colon Z\to\mathbb{R}$ satisfying
(\ref{holder}), where $Z$ is $X$ or $Y$ depending
on the context.
\begin{proposition}
Given $\alpha>0$ and $\gamma\in (0,1]$ there exists $C>0$ such
that
\[
  \left|\int \psi\, d\nu_{\phi_1}-\int \psi\, d\nu_{\phi_2}\right|\le C \|\phi_1-\phi_2\|,\quad
  \text{for every } \psi,\phi_1,\phi_2\in\mathcal{H}^{\alpha,\gamma},
\]
where $\nu_{\phi}$ is the Gibbs state for the potential $\phi$.
\end{proposition}
\begin{proof}
Write $\phi_2=\phi_1+\delta u$ where
$u=(\phi_2-\phi_1)/\|\phi_1-\phi_2\|$ and $\delta=\|\phi_1-\phi_2\|$.
Then, by \cite{R1}, we have that
\begin{align*}
  \left|\int \psi\, d\nu_{\phi_1}-\int \psi\, d\nu_{\phi_2}\right|&=
  \left|\frac{\partial P(\phi_1+t\psi)}{\partial t}-
  \frac{\partial P(\phi_1+\delta u+t\psi)}{\partial t}\right|_{t=0}\\
  &=\delta \left|\frac{\partial^2{P(\phi_1+su+t\psi)}}{\partial t \partial s}\right|_{t=0,s=\xi}
  =\delta Q_{\phi_1}(\psi,\xi u),
\end{align*}
for some $\xi\in[0,\delta]$, where $Q$ is the bilinear form
defined by (\ref{forma}). Now it is well known (see Chapter 5, Exercise 4 of \cite{R1}) that
the general term of the series defining $Q$
decreases at an exponential rate $c\theta^n$ for some $c=c(\alpha,
\gamma)>0$ and $0<\theta=\theta(\alpha,\gamma)<1$. So take
$C=c/(1-\theta)$.
\end{proof}

The following result follows from Theorem 2.10 and Equation (5.8) of \cite{DG} applied
to the constant function 1.

\begin{proposition}
For any $\gamma\in (0,1]$ there exist $\eta(\gamma)\in (0,1]$ and
$C_\gamma>0$ such that if $\phi\in\mathcal{H}^{D,\gamma}$ then
$A_\phi\in\mathcal{H}^{C_\gamma D \exp(\|\phi\|),\eta(\gamma)}$.
\end{proposition}
Let
\[
   F(t,\beta)=\int \log A_{-t\varphi}\,d\nu_{(t,\beta)},
\]
where $\nu_{(t,\beta)}$ is the Gibbs sate for the potential
$(t-D)\psi+\beta \log A_{-t\varphi}$. By Propositions 1 and 2
together with the continuity of $[0,1]\ni t\mapsto A_{-t\varphi}$,
we conclude that $(t,\beta)\mapsto F(t,\beta)$ is continuous.
Also, by \cite{R1},
\begin{equation}\label{imp1}
    \frac{\partial F}{\partial \beta}(t,\beta)=Q_{(t-D)\psi+\beta \log A_{-t\varphi}}
    (\log A_{-t\varphi}, \log A_{-t\varphi})>0,
\end{equation}
and, for each $t$, $\beta\mapsto \frac{\partial F}{\partial
\beta}(t,\beta)$ is continuous. Let $t_0\in
(\underline{t},\overline{t})$ and $\beta_0=\beta(t_0)$. Then
$F(t_0,\beta_0)=0$ and given $\varepsilon>0$ sufficiently small we
get by (\ref{imp1}) that
\begin{equation}\label{imp2}
   \text{for } t=t_0,\quad F(t,\beta_0-\varepsilon)<0 \quad\text{and}\quad
   F(t,\beta_0+\varepsilon)>0.
\end{equation}
By continuity, there is $\delta>0$ such that (\ref{imp2}) holds
for all $t\in (t_0-\delta, t_0+\delta)$. So, by the intermediate
value theorem, there is a unique $\tilde{\beta}(t)\in
(\beta_0-\varepsilon, \beta_0+\epsilon)$ such that
$F(t,\tilde{\beta}(t))=0$ for all $t\in (t_0-\delta, t_0+\delta)$.
By uniqueness we have $\beta(t)=\tilde{\beta}(t)$, which implies
that $\beta(t)$ is continuous at $t_0$.

Finally, the continuity of $h(t)$ follows from the continuity of
$t\mapsto A_{-t\varphi}$ and $\beta(t)$, together with
$|P(\varphi_1)-P(\varphi_2)|\le \|\varphi_1-\varphi_2\|$ (see \cite{R1}).
So if the supremum of $(\underline{t},\overline{t})\ni t\mapsto
h(t)$ is attained at $t^*\in (\underline{t},\overline{t})$ then
\[
     D=D(\mu_{\nu_{\beta(t^*)}}).
\]

To finish, we must consider the cases for which the supremum in (\ref{novopv2}) is attained at
$\underline{t}$ or $\overline{t}$. If the
supremum is attained at $\underline{t}$ then
\begin{equation*}
 D\le \sup_{\nu\in\mathcal{M}(S)} \frac{h_\nu(S)}{\int \psi\,d\nu} +\underline{t}
 \le \frac{h_{\nu_0}(S)}{\int \psi\,d\nu_0}+t(\nu_0)\le D,
\end{equation*}
where $\nu_0$ is the Gibbs measure for the potential $-D\psi$,
so equality holds. Now assume that the supremum is attained at
$\overline{t}$. Then by (\ref{novopv1}) and standard upper-semicontinuous
arguments we get a measure $\nu\in\mathcal{M}(S)$ such that
\begin{equation}\label{extremo}
   h_\nu(S)+(\overline{t}-D)\int \psi\,d\nu=0\quad\text{and}\quad
   \int \log A_{-\overline{t}\varphi}\,d\nu=0.
\end{equation}
Let $\nu=\int \nu_\alpha\,d\alpha$ be the ergodic decomposition of
$\nu$. Since $\int \log A_{-\overline{t}\varphi}\,d\nu_\alpha\le0$
(because $t(\nu_\alpha)\le\overline{t}$) and $\int\int \log
A_{-\overline{t}\varphi}\,d\nu_\alpha\,d\alpha=0$, we get that
$\int \log A_{-\overline{t}\varphi}\,d\nu_\alpha=0$ i.e
$t(\nu_\alpha)=\overline{t}$ for a.e $\alpha$. Then using the
formula for the ergodic decomposition for the entropy in
(\ref{extremo}) we get
\[
  \int \left(h_{\nu_\alpha}(S)+ (t(\nu_\alpha)-D)\int \psi\,d\nu_\alpha \right) \,d\alpha=0.
\]
Since $h_{\nu_\alpha}(S)+ (t(\nu_\alpha)-D)\int \psi\,d\nu_\alpha\le0$ for every $\alpha$,
we get that
\[
   h_{\nu_\alpha}(S)+ (t(\nu_\alpha)-D)\int \psi\,d\nu_\alpha=0
\]
for a.e. $\alpha$, and thus $D=D(\mu_{\nu_\alpha})$.
\begin{flushright}$\square$\end{flushright}

\begin{rem}
By the proof of Theorem A, if the supremum (\ref{novopv2}) is not attained at the extremal points
$\underline{t}$ and $\overline{t}$, then the maximizing measures of (\ref{pvprinc}) are given by
$\mu_{\nu_{\beta(t)}}$ where $t$ is a zero (i.e. a maximizing point) of the function $h$ defined
by (\ref{h}). So if $h$ is $\mathrm{C}^2$ and $h''<0$ then (\ref{pvprinc}) has a \emph{unique}
maximizing measure, which is ergodic.

Note that, by (\ref{rpv}), (\ref{gauge}) and the classical variational principle,
\begin{equation*}
 P(\log A_\varphi + \psi, Y)=P(\varphi+\psi\circ\pi, X).
\end{equation*}
Is there some kind of relation for $\beta \log A_\varphi + \psi, \beta\in\mathbb{R}$?
\end{rem}

\begin{rem}
In fact, $(T,X)$ need not be a mixing subshift of finite type; what we really need is
$\mathcal{F}=(X,T,Y,S,\pi)$ to be a \emph{fibred system} which is \emph{fibre expanding} and
\emph{topologically exact along fibres} as defined in \cite{DG} or \cite{DGH}.
\end{rem}

\section{Measure of full dimension}

Let $f\colon \mathbb{T}^2\to\mathbb{T}^2,\,f(x,y)=(a(x,y),b(y))$ and $\Lambda$ such that
$f(\Lambda)=\Lambda$ be as in Theorem B. Let $\pi\colon\mathbb{T}^2\to\mathbb{T}^1$ be the
projection given by $\pi(x,y)=y$. Then $\pi\circ f= b\circ\pi$, and we are in the conditions of
Theorem A with $T\equiv f|\Lambda$, $S\equiv b|\pi(\Lambda)$, $\varphi\equiv \log\partial_x a$ and
$\psi\equiv\log b'$.
(We have used ``$\equiv$'' instead of ``$=$'' because to be more precise one should
use conjugacies to identify these maps; these conjugacies are continuous, surjective and
bounded-to-one, and only fail to be a homeomorphism when
some elements of the Markov partition intersect, which causes no problem when dealing with
Hausdorff dimension.)

Then, using the notation of the previous section, the following is proved in \cite{L1}.

\begin{theorem}\label{teo2}
\[
   \hd\Lambda=\sup_{\nu\in\mathcal{M}_{e}(b|\pi(\Lambda))} \hd \mu_\nu
\]
and if $\nu\in\mathcal{M}_{e}(b|\pi(\Lambda))$ then
\[
  \hd \mu_\nu=\frac{h_\nu(b)}{\int \log b'\,d\nu}+t(\nu).
\]
\end{theorem}

\begin{rem}
Even though the map
\[
  \mathcal{M}_{e}(b|\pi(\Lambda))\ni\nu\mapsto \hd\mu_\nu
\]
is upper-semicontinuous (see Remark 6 of \cite{L1}), we cannot conclude
there is an invariant measure of full dimension because the subset
$\mathcal{M}_{e}(b|\pi(\Lambda))\subset\mathcal{M}(b|\pi(\Lambda))$
is not closed.
\end{rem}
\text{}

\emph{Proof of Theorem B}.
By the proof of Theorem A,
\[
  \sup_{\nu\in\mathcal{M}_e(b|\pi(\Lambda))} \left\{ \frac{h_\nu(b)}{\int \log b'\,d\nu}+t(\nu)
  \right\}
\]
is attained at some $\nu_0\in\mathcal{M}_e(b|\pi(\Lambda))$. Then, by Theorem \ref{teo2},
\[
   \hd\Lambda=\hd\mu_{\nu_0}.
\]
\begin{flushright}$\square$\end{flushright}

\begin{rem}
It follows from \cite{LY} that, given $\mu\in\mathcal{M}_e(f|\Lambda)$ with $\nu=\mu\circ\pi^{-1}$,
\[
  \hd\mu=\frac{h_\nu(b)}{\int \log b'\,d\nu}+\frac{h_\mu(f)-h_\nu(b)}{\int\log\partial_x a\,d\mu}.
\]
Now by (\ref{restrain}),
\[
  \frac{h_\mu(f)-h_\nu(b)}{\int\log\partial_x a\,d\mu}\le t(\nu),
\]
so
\[
  \hd\mu\le\frac{h_\nu(b)}{\int \log b'\,d\nu}+t(\nu)=\hd\mu_\nu
\]
with equality iff $\mu=\mu_\nu$.
This shows that ergodic invariant measures of full dimension are obtain by our method.
\end{rem}

\textbf{Acknowledgments:} I would like to thank a referee of a previous version of this manuscript
for careful reading and helpful suggestions.
\\\\

\text{}\\
\textsc{Departamento de Matemática, Instituto Superior Técnico,
1049-001 Lisboa, Portugal}

\emph{E-mail address}: nluzia@math.ist.utl.pt

\end{document}